\documentclass[11pt]{amsart}
\usepackage{graphicx}
\usepackage[active]{srcltx}
 \makeatletter
\renewcommand*\subjclass[2][2000]{%
  \def\@subjclass{#2}%
  \@ifundefined{subjclassname@#1}{%
    \ClassWarning{\@classname}{Unknown edition (#1) of Mathematics
      Subject Classification; using '1991'.}%
  }{%
    \@xp\let\@xp\subjclassname\csname subjclassname@#1\endcsname
  }%
}
 \makeatother
\usepackage{enumerate,url,amssymb,  mathrsfs}

\newtheorem{theorem}{Theorem}[section]
\newtheorem{lemma}[theorem]{Lemma}
\newtheorem*{lemma*}{Lemma}

\newtheorem{corollary}[theorem]{Corollary}

\theoremstyle{definition}

\newtheorem{conjecture}[theorem]{Conjecture}

\theoremstyle{remark}
\newtheorem{remark}[theorem]{Remark}

\numberwithin{equation}{section}

\def\XXint#1#2#3{{\setbox0=\hbox{$#1{#2#3}{\int}$}
\vcenter{\hbox{$#2#3$}}\kern-.5\wd0}}

\def\le{\leqslant}
\def\ge{\geqslant}
\begin{document}

\title{A proof of Khavinson's conjecture in $\mathbf{R}^4$}

\author{David Kalaj}
\address{University of Montenegro, Faculty of Natural Sciences and
Mathematics, Dzordza Va\v singtona  b.b. 81000 Podgorica, Montenegro}
\email{davidkalaj@gmail.com}

\begin{abstract}
The paper deals with an extremal problem for bounded  harmonic functions in the unit
ball of $\mathbf{R}^4$.   We solve the generalized Khavinson    problem in $\mathbf{R}^4$.   This precise problem was formulated   by G. Kresin and V. Maz'ya
for harmonic functions in the unit ball and in the half--space of $\mathbf{R}^n$.  We find
the optimal pointwise   estimates   for   the norm of the gradient of  bounded real--valued harmonic functions.
\end{abstract}
\maketitle

\section{Introduction and statement of the main results}

    In   this   paper we consider the sharp pointwise
estimates for the  gradients of real--valued bounded harmonic functions.  We will first recall the  known
estimates  of this type in the plane and in the space.

For every  fixed            $z=(x,y)\in\mathbf{R}^2_+$ there holds the     optimal gradient estimate
\begin{equation}\label{INEQ.2}
|\nabla V(z)|\le\frac 2\pi \frac 1{y} |V|_\infty,
\end{equation}
where  $V$  is  an arbitrary  bounded  harmonic  functions in  the  upper  half--plane $H^2=\mathbf{R}^2_+$, and $|V|_\infty=\sup_{z\in \mathbf{R}^2_+}|V(z)|$.
Using  the conformal  transformation of the unit disk $\mathbf{B}^2$ onto $\mathbf{R}^2_+$ one easily derives
\begin{equation}\label{INEQ.CLASSICAL}
|\nabla U(z)|\le \frac 4\pi \frac1{1-|z|^2}|U|_\infty,
\end{equation}
for    $x\in \mathbf{B}^2$,  where $U$    is  harmonic  and bounded  in the unit disc \cite{COLONNA.INDIANA}.    This classical
result is  improved in the recent paper of D. Kalaj and M. Vuorinen \cite{KALAJ.PAMS}. Their form of
the above  inequality  says that
\begin{equation}\label{INEQ.KALAJ}
|\nabla U(z)|\le \frac 4\pi \frac{1-U(z)^2}{1-|z|^2}.
\end{equation}
This relation requires that $U$ is  bounded by $1$ in the disc $\mathbf{B}^2$.     The proof of   the
inequality \eqref{INEQ.KALAJ} lies on the classical Schwarz lemma for bounded analytic
functions.

Recently G.   Kresin and
V. Maz'ya \cite{KRESIN.DCDS}      proved the   following  generalization od \eqref{INEQ.2}:
\begin{equation}\label{EST.KR.MA.H}
|\nabla V (x)|\le
\frac 4 {\sqrt{\pi}} \frac {(n-1)^{ (n-1)/2} } { n^{n/2} }\frac {\Gamma(n/2)} {\Gamma((n-1)/2)}
\frac 1{x_n} |V|_\infty.
\end{equation}
Here,   $V$ is a bounded harmonic function in the half--space $\mathbf{R}^n_+$, $|V|_\infty=\sup_{y\in \mathbf{R}^n_+}|V(y)|$  , and $x = (x',x_n)\in
\mathbf{R}^n_+$  is fixed. These  optimal poitwise  estimates arise arise while proving Khavinson conjecture in the halfspace setting. In order to  formulate the conjecture we introduce
the notation we need.

For every    fixed $x$ let $\mathcal{C}(x)$   denote      the optimal number for the gradient  estimate
\begin{equation}\label{FIRST.EST}
\left|\nabla  U(x)\right|\le \mathcal{C}(x)|U|_\infty,
\end{equation}
where  $U$ is  harmonic and  bounded in $\mathbf{B}^n$ or $\mathbf{R}^n_+$.        Similarly, for every
$\mathbf{v}\in \mathbb{R}^n,\, |\mathbf {v}|=1$  denote  by  $\mathcal{C}(x,\mathbf{v})$    the optimal
number for the gradient estimate in the direction $\mathbf{v}$, i.e., the smallest number such that the
following   relation holds
\begin{equation*}
\left|\left< \nabla U(x), \mathbf {v}\right> \right|\le \mathcal{C}(x,\mathbf{v}) |U|_\infty
\end{equation*}
for every bounded and harmonic  $U$. Since
\begin{equation*}
\left|\nabla U(x)\right| = \sup_{\mathbf{v}\in\partial \mathbf{B}^n} \left|\left<\nabla U(x), \mathbf{v}\right>\right|,
\end{equation*}
it follows that
\begin{equation}\label{VAR.PROBLEM}
\mathcal{C}(x)=\sup_{\mathbf {v}}   \mathcal{C}(x,\mathbf {v}).
\end{equation}
It turned out that the variational problem \eqref{VAR.PROBLEM} is a hard problem,           especially
in the unit ball setting.                            The generalized  Khavinson conjecture states that

\begin{conjecture}\label{CONJ.KH} For $x\in\mathbf{B}^n$ we have
\begin{equation}\label{khav}
\mathcal{C}(x) = \mathcal {C} (x,\mathbf{n}_x),
\end{equation}
where $\mathbf{n}_x=x/|x|$ is the vector normal  to the boundary at $x$.
\end{conjecture}

In        1992, Khavinson \cite{KHAVINSON.CMB} obtained  the optimal  estimate in the normal direction
of the gradient  of  bounded harmonic functions in the unit ball in    $\mathbf{R}^3$.  In a   private
conversation with  K. Gresin and V. Maz'ya  he believed that  the same estimate  hold for the
norm of the gradient, i.e., that the above conjecture      is true for  the  unit ball $\mathbf {B}^3$.

In their    recent  paper \cite{KRESIN.DCDS} and in their book \cite{book}, G. Kresin and V. Maz'ya considered the Khavinson  problem
from a more general aspect including harmonic functions        with $L^p$-boundary values ($1\le p \le
\infty$). They formulated the    generalized Khavinson conjecture and proved it  for bounded
harmonic functions in  $\mathbf{R}^n_+$. In this context  we  have $\mathbf{n}_x = \mathbf{e}_n$
for all   $x\in \mathbf{R}^n_+$. After replacing $\mathcal{C}(x)$ with   $\mathcal{C}(x,\mathbf{e}_n)$
in \eqref{FIRST.EST},  they  obtained \eqref{EST.KR.MA.H}.

 M. Markovi\'c in a recent paper \cite{MARKOVIC.CA} proved the conjecture when $x$ is near the boundary, i.e., if
$1-\epsilon\le |x|< 1$. Therefore, in \eqref{FIRST.EST} one can  replace $\mathcal{C} (x)$ with
$\mathcal{C}(x,\mathbf{n}_x)$, if  $|x|$ is near $1$. In this paper we prove the conjecture for $n=4$, i.e. we prove the following theorem
\begin{theorem}\label{kwer}
For $x\in\mathbf{B}^4$ we have
\begin{equation}\label{khav4}
\mathcal{C}(x) = \mathcal {C} (x,\mathbf{n}_x),
\end{equation}
where $\mathbf{n}_x=x/|x|$ is the vector normal  to the boundary at $x$.
\end{theorem}

A reformulated version of Theorem~\ref{kwer} is the following theorem, whose proof follows directly from
 Theorem~\ref{help} and relation \eqref{markovic} below.
\begin{theorem}
Let $n=4$. Then we have the sharp inequality for every $x\in \mathbf{B}^4$, $r=|x|$: $$|\nabla u(x)|\le \frac{ \left(r \sqrt{4-r^2} \left(2+r^2\right)+4 \left(1-r^2\right) \tan^{-1}\left[r\frac{\sqrt{4-r^2}}{r^2-2}\right]\right)}{\pi(1-r^2)r^3}|u|_\infty, \ \ \ u\in h^\infty(\mathbf{B}^4).$$ Here and in the sequel,  $h^\infty(\mathbf{B}^4)$ is the Hardy space of bounded harmonic functions on the unit ball  $\mathbf{B^4}$ (cf. \cite{AXLER.BOOK}).
\end{theorem}

\begin{corollary}
 For the decreasing diffeomorphism  $\frak{C}:[0,1]\to \left[\frac{3 \sqrt{3}  }{2\pi},\frac{16}{3\pi}\right]$, defined by $$\frak{C}(r)=\frac{ \left(r \sqrt{4-r^2} \left(2+r^2\right)+4 \left(1-r^2\right) \tan^{-1}\left[\frac{r \sqrt{4-r^2}}{-2+r^2}\right]\right)}{\pi r^3 },$$ we have the sharp inequality for every $x\in \mathbf{B}^4$, $r=|x|$: $$|\nabla u(x)|\le \frac{\frak{C}(r)}{1-r^2}|u|_\infty\ \ \ u\in h^\infty(\mathbf{B}^4).$$
\end{corollary}
\begin{remark}
Observe that for $\mathbf{R}^4_+$, Kresin - Maz'ya inequality \eqref{EST.KR.MA.H} reads as
\begin{equation}\label{EST.KR.MA.H1}
|\nabla V (x)|\le
\frac{3 \sqrt{3}  }{2\pi}
\frac 1{x_4} |V|_\infty.
\end{equation}
\end{remark}
\begin{proof}[Proof of corollary]
 We need to prove that $\frak{C}(r)$ is a decreasing function. We have that $$\frak{C}'(r)=-\frac{  \left((-2+r) r (2+r) \left(-6+r^2\right)-4 \sqrt{4-r^2} \left(-3+r^2\right) \tan^{-1}\left[\frac{r \sqrt{4-r^2}}{-2+r^2}\right]\right)}{\pi r^4 \sqrt{4-r^2}}.$$ So $\frak{C}'(r)\le 0$ if and only if $$v(r)=\frac{r(4-r^2)\left(6-r^2\right)}{4 \left(3-r^2\right) \sqrt{4-r^2}}+\tan^{-1}\left[\frac{r \sqrt{4-r^2}}{-2+r^2}\right]\ge 0.$$ Since $$v'(r)=\frac{r^4 \sqrt{4-r^2}}{2 \left(3-r^2\right)^2}\ge 0,$$ and $v(0)=0$ and the claim follows.
\end{proof}

\section{The technical lemmas}
Let $r=|x|$. For $n\ge 3$, let $\omega_{n}$ be the area of $S^{n-1}$. Markovi\'c in \cite{MARKOVIC.CA} proved that \begin{equation}\label{markovic}\mathcal{C}(x)=\frac{1}{1-r}\sup_{z>0} C(z,r),\end{equation} where \begin{equation}\label{crz}C(z,r)=\frac{4\omega_{n-2}}{\omega_n}\frac{2^{n-1}}{(1+r)^{n-1}}\frac{1}{\sqrt{1+z^2}}\int_0^1\frac{\Psi_r(z t)+\Psi_r(-z t)}{\sqrt{(1-t^2)^{4-n}}}dt.\end{equation} Here
\begin{equation}\label{psi}\Psi_r(z)=\int_0^{\frac{z+\sqrt{z^2+1-\alpha^2_r}}{1-\alpha_r}}
\frac{n-\beta_r+nz w-\beta_r w^2}{(1+w^2)^{n/2+1}(1+k_r^2 w^2)^{n/2-1}} w^{n-2} dw, \end{equation} and $$k_r = \frac{1-r}{1+r},\,\, \alpha_r = \frac{r (n - 2)}{n},\,\, \beta_r = \frac{(n - (n - 2) r)}{2}.$$
Further, in the same paper he has showed that the conjectured equality \eqref{khav} is equivalent to the equality \begin{equation}\label{kalaj} \sup_{z>0} C(z,r)=C(0,r).\end{equation} Our goal is to prove \eqref{kalaj} for $n=4$.
\subsection{Explicit representation of $\Psi$ for $n=4$}
 Let us recalculate the integrand in \eqref{psi}:

\[\begin{split}Q(w)&=\frac{w^2 \left(4+\frac{1}{2} (-4+2 r)-\frac{1}{2} (4-2 r) w^2+4 w z\right)}{\left(1+w^2\right)^3 \left(1+\frac{(1-r)^2 w^2}{(1+r)^2}\right)}\\&=\frac{(1+r)^2 w^2 \left(2+r-2 w^2+r w^2+4 w z\right)}{\left(1+w^2\right)^3 \left((1+r)^2+(-1+r)^2 w^2\right)}\\&=\frac{(1+r)^2}{r \left(1+w^2\right)^3}-\frac{(1+r)^2 (1+4 r)}{4 r^2 \left(1+w^2\right)^2}\\&+\frac{(1+r)^4}{16 r^3 \left(1+w^2\right)}-\frac{(-1+r)^2 (1+r)^4}{16 r^3 \left((1+r)^2+(-1+r)^2 w^2\right)}\\&+ \frac{(1+r)^2 zw}{r \left(1+w^2\right)^3}-\frac{(1+r)^4 zw}{4 r^2 \left(1+w^2\right)^2}\\&+\frac{(-1+r)^2 (1+r)^4 zw}{16 r^3 \left(1+w^2\right)}-\frac{\left(-1+r^2\right)^4 zw}{16 r^3 \left((1+r)^2+(-1+r)^2 w^2\right)}.\end{split}\]
By elementary integration and since
$$\int \frac{1}{(1+w^2)^3}dw=\frac{1}{8} \left(\frac{w \left(5+3 w^2\right)}{\left(1+w^2\right)^2}+3 \tan^{-1}[w]\right)$$ while

$$\int \frac{1}{(1+w^2)^2} dw =\frac{1}{2} \left(\frac{w}{1+w^2}+\tan^{-1}[w]\right)$$
we obtain
\begin{equation}\label{split}\begin{split}R(w)&=\frac{32 r^3}{(1+r)^2}\int Q(w)dw\\&= \frac{4 r w \left(1+w^2+r \left(-1+w^2\right)\right)-4 r \left(1+r^2+(1+r)^2 w^2\right) z}{\left(1+w^2\right)^2}\\&+2 \left(-1+r^2\right) \tan^{-1}[w]+2 \left(-1+r^2\right) \tan^{-1}\left[\frac{(-1+r) w}{1+r}\right]\\&+\left(-1+r^2\right)^2 z \log\left[\frac{(1+r)^2+(-1+r)^2 w^2}{1+w^2}\right].\end{split}\end{equation} Thus we have
\begin{lemma}\label{le0} For $r\in(0,1)$ and $z>0$ we have
\[\begin{split}&\Psi_r(z)=\frac{(1-r) (1+r)^3}{64 r^3 }\times\\& \bigg( \frac{r\left(\left(4+r^2 (4+r)\right) z+4 \left(1+r^2\right) z^3+\left(2+r^2+2 \left(1+r^2\right) z^2\right) \sqrt{4-r^2+4 z^2}\right)}{\left(1+z^2\right)\left(1-r^2\right) }\\&+4  \tan^{-1}\left[\frac{r \left(-2 r z+(r^2-2) \sqrt{4-r^2+4 z^2}\right)}{\left(-2+r^2\right)^2-4 \left(-1+r^2\right) z^2}\right]\\&+2 \left(1-r^2\right)   z \log\left[\frac{1+z \left(z+r^2 z-r \sqrt{4-r^2+4 z^2}\right)}{(1+r)^2 \left(1+z^2\right)}\right]\bigg).\end{split}\]
\end{lemma}
\begin{proof}
In view of \eqref{split} and \eqref{psi} we obtain $$\frac{32 r^3}{(1+r)^2}\Psi_r(z)=\frac{32 r^3}{(1+r)^2}\int_0^{\frac{z+\sqrt{z^2+1-\alpha^2_r}}{1-\alpha_r}} Q(w)dw=R\left(\frac{z+\sqrt{z^2+1-\alpha^2_r}}{1-\alpha_r}\right)-R(0),$$ which after some elementary transformations implies  the lemma.
\end{proof}
\subsection{Explicit representation of $C$ for $n=4$}
From Lemma~\ref{le0} we have
\begin{lemma} For $r\in(0,1)$ and $z>0$ we have
\[\begin{split}\Psi_r(zt)&+\Psi_r(-zt)=\frac{(1-r) (1+r)^3}{16 r^3}\bigg (\frac{r \sqrt{4-r^2+4 t^2 z^2} \left(2+r^2+2 \left(1+r^2\right) t^2 z^2\right)}{2 \left(1-r^2\right) \left(1+t^2 z^2\right)}\\&- \tan^{-1}\left[\frac{r \left(2 r t z+(2-r^2) \sqrt{4-r^2+4 t^2 z^2}\right)}{\left(-2+r^2\right)^2-4 \left(-1+r^2\right) t^2 z^2}\right]\\&+ \tan^{-1}\left[\frac{r \left(2 r t z-(2-r^2) \sqrt{4-r^2+4 t^2 z^2}\right)}{\left(-2+r^2\right)^2-4 \left(-1+r^2\right) t^2 z^2}\right]\\&- \left(1-r^2\right) t z  \tanh^{-1}\left[\frac{r t z \sqrt{4-r^2+4 t^2 z^2}}{1+\left(1+r^2\right) t^2 z^2}\right]\bigg).\end{split}\]
\end{lemma}
Using integration by parts for $V=t$ and \[\begin{split}U&=-\tan^{-1}\left[\frac{r \left(2 r t z+(2-r^2) \sqrt{4-r^2+4 t^2 z^2}\right)}{\left(-2+r^2\right)^2-4 \left(-1+r^2\right) t^2 z^2}\right]\\&+\tan^{-1}\left[\frac{r \left(2 r t z-(2-r^2) \sqrt{4-r^2+4 t^2 z^2}\right)}{\left(-2+r^2\right)^2-4 \left(-1+r^2\right) t^2 z^2}\right],\end{split}\] and in view of the formula $$U'=\frac{r t z^2 \left(\left(-2+r^2\right)^2+2 \left(2-3 r^2+r^4\right) t^2 z^2\right)}{\left(1+t^2 z^2\right) \sqrt{4-r^2+4 t^2 z^2} \left(1+\left(-1+r^2\right)^2 t^2 z^2\right)},$$ which can be proved by a direct computation, we obtain

\[\begin{split}\int U(t)dV&=t U(t)-\int t\frac{r t z^2 \left(\left(-2+r^2\right)^2+2 \left(2-3 r^2+r^4\right) t^2 z^2\right)}{\left(1+t^2 z^2\right) \sqrt{4-r^2+4 t^2 z^2} \left(1+\left(-1+r^2\right)^2 t^2 z^2\right)}dt.\end{split}\]
Similarly, for $V_1=t^2/2$ and $$U_1= -z(1-r^2)\tanh^{-1}\left[\frac{r t z \sqrt{4-r^2+4 t^2 z^2}}{1+\left(1+r^2\right) t^2 z^2}\right] $$ we obtain
$$U_1'=\frac{r z^2 \left(-4+5 r^2-r^4+\left(-4+7 r^2-4 r^4+r^6\right) t^2 z^2\right)}{\left(1+t^2 z^2\right) \sqrt{4-r^2+4 t^2 z^2} \left(1+\left(-1+r^2\right)^2 t^2 z^2\right)},$$ and then \[\begin{split}\int U_1(t)dV_1&=U_1 V_1-\int \frac{t^2}{2} \frac{r z^2 \left(-4+5 r^2-r^4+\left(-4+7 r^2-4 r^4+r^6\right) t^2 z^2\right)}{\left(1+t^2 z^2\right) \sqrt{4-r^2+4 t^2 z^2} \left(1+\left(-1+r^2\right)^2 t^2 z^2\right)} dt.\end{split}\]
 Furthermore we have $$VU'+V_1U_1'=-\frac{r t^2 z^2 \left(-4+3 r^2-r^4-\left(4-5 r^2+r^6\right) t^2 z^2\right)}{2 \left(1+t^2 z^2\right) \sqrt{4-r^2+4 t^2 z^2} \left(1+\left(-1+r^2\right)^2 t^2 z^2\right)}.$$
Let $$Y = \frac{r \sqrt{4-r^2+4 t^2 z^2} \left(2+r^2+2 \left(1+r^2\right) t^2 z^2\right)}{2 \left(1-r^2\right) \left(1+t^2 z^2\right)}$$ and 
$$X=-VU'-V_1U_1'+Y.$$
By using the formula
$$\int\frac{a }{\sqrt{b+c t^2} \left(1+\frac{\left(c-a^2\right)}{b} t^2\right)}dt=\tanh^{-1}\left[\frac{a t}{\sqrt{b+c t^2}}\right]$$ and the representation

\[\begin{split}{2 \left(1-r^2\right) z}X&=\frac{4 r \left(1+r^2\right) t^2 z^3}{\sqrt{4-r^2+4 t^2 z^2}}+r \left(1+r^2\right) z \sqrt{4-r^2+4 t^2 z^2}\\&+\frac{r z}{ \sqrt{4-r^2+4 t^2 z^2}\left(1+t^2 z^2\right)}-\frac{r \left(-3+r^2\right) z}{\sqrt{4-r^2+4 t^2 z^2} \left(1+\left(-1+r^2\right)^2 t^2 z^2\right)},\end{split}\] we obtain
$$\int X dt  =\frac{r \left(1+r^2\right) t z \sqrt{4-r^2+4 t^2 z^2}+\tanh^{-1}\left[\frac{r t z}{\sqrt{4-r^2+4 t^2 z^2}}\right]-\tanh^{-1}\left[\frac{r \left(-3+r^2\right) t z}{\sqrt{4-r^2+4 t^2 z^2}}\right]}{2 \left(1-r^2\right) z}.$$
Now  the formula \[\begin{split}\frac{16 r^3}{(1-r) (1+r)^3}\left(\int \Psi_r(zt)+\Psi_r(-zt)dt\right)&=\left(\int Y dt + \int U dv +\int U_1 dv_1\right)\\&=\int Y dt+UV+U_1V_1-\int VdU-\int V_1 dU_1\\&=UV+U_1V_1+\int X dt,\end{split}\] yields the following
\begin{lemma}\label{poku} For $t\in[0,1]$ we have
\[\begin{split}\int_0^t(\Psi_r(zs)&+\Psi_r(-zs))ds=\frac{1}{16 r^3 z}(1+r)^2 \bigg(\frac{1}{2} r \left(1+r^2\right) t z \sqrt{4-r^2+4 t^2 z^2}\\&+ \frac{1}{2}{\tanh^{-1}\left[\frac{r t z}{\sqrt{4-r^2+4 t^2 z^2}}\right]-\frac{1}{2}\tanh^{-1}\left[\frac{r \left(-3+r^2\right) t z}{\sqrt{4-r^2+4 t^2 z^2}}\right]}\\&+ \left(-1+r^2\right) t z \tan^{-1}\left[\frac{r \left(2 r t z+(2-r^2) \sqrt{4-r^2+4 t^2 z^2}\right)}{\left(-2+r^2\right)^2-4 \left(-1+r^2\right) t^2 z^2}\right]\\&- \left(-1+r^2\right) t z \tan^{-1}\left[\frac{r \left(2 r t z-(2-r^2) \sqrt{4-r^2+4 t^2 z^2}\right)}{\left(-2+r^2\right)^2-4 \left(-1+r^2\right) t^2 z^2}\right]\\&-\frac{1}{2} \left(-1+r^2\right)^2 t^2 z^2 \tanh^{-1}\left[\frac{r t z \sqrt{4-r^2+4 t^2 z^2}}{1+\left(1+r^2\right) t^2 z^2}\right]\bigg).\end{split}\]
\end{lemma}
By putting $t=1$ in Lemma~\ref{poku} and using \eqref{crz}, in view of  $$\frac{4\omega_{4-2}}{\omega_4}\frac{2^{4-1}}{(1+r)^{4-1}}=\frac{4\cdot 2\pi \cdot 8}{2\pi^2 (1+r)^3}=\frac{32}{\pi (1+r)^3}, $$ we obtain
\begin{lemma}\label{l1} For $r\in(0,1)$ and $z>0$ we have $$C(r,z)=\frac{  2(1-r)}{\pi r^3 \sqrt{1+z^2}} (h_1(z)+h_2(z)+h_3(z)),$$ where

$$h_1(z)=\frac{r \left(1+r^2\right)  \sqrt{4-r^2+4 z^2}}{2 \left(1-r^2\right) } +\frac{\tanh^{-1}\left[\frac{r z}{\sqrt{4-r^2+4 z^2}}\right]-\tanh^{-1}\left[\frac{r \left(-3+r^2\right) z}{\sqrt{4-r^2+4 z^2}}\right]}{2 \left(1-r^2\right) z},$$
\[\begin{split}h_2(z)&=\tan^{-1}\left[\frac{r \left(2 r z-(2-r^2) \sqrt{4-r^2+4 z^2}\right)}{\left(-2+r^2\right)^2-4 \left(-1+r^2\right)z^2}\right]\\&-\tan^{-1}\left[\frac{r \left(2 r z+(2-r^2) \sqrt{4-r^2+4 z^2}\right)}{\left(-2+r^2\right)^2-4 \left(-1+r^2\right) z^2}\right] \end{split}\]
and $$h_3(z)= \frac{1}{2}\left(-1+r^2\right) z \tanh^{-1}\left[\frac{r z \sqrt{4-r^2+4 z^2}}{1+\left(1+r^2\right) z^2}\right].$$

\end{lemma}
Finally we need the following lemmata
\begin{lemma}\label{l2}
Let $$L(z)=\tanh^{-1}\left[\frac{r z}{\sqrt{4-r^2+4 z^2}}\right]-\tanh^{-1}\left[\frac{r \left(-3+r^2\right) z}{\sqrt{4-r^2+4 z^2}}\right].$$ Then, $L(z)\le  r z \sqrt{4-r^2}$.
\end{lemma}
\begin{proof}
By differentiating $L$ we obtain $$L'(z)=\frac{ r \left(4-r^2+\left(4-3 r^2+r^4\right) z^2\right)}{\left(1+z^2\right) \sqrt{4-r^2+4 z^2} \left(1+\left(-1+r^2\right)^2 z^2\right)}.$$
Since $$\partial_z \frac{\left(4-r^2+\left(4-3 r^2+r^4\right) z^2\right)}{\left(1+z^2\right)\text{  }\left(1+\left(-1+r^2\right)^2 z^2\right)}=-\frac{2 z}{\left(1+z^2\right)^2}+\frac{2 \left(-3+r^2\right) \left(-1+r^2\right)^2 z}{\left(1+\left(-1+r^2\right)^2 z^2\right)^2}\le 0,$$ it follows that $(L(z)- z r \sqrt{4-r^2})'\le L'(0)- r \sqrt{4-r^2}=0$. So $L(z)\le z r \sqrt{4-r^2}$.
\end{proof}
\begin{lemma}\label{l3}
Let \[\begin{split}g_1(z)&=\frac{ r \sqrt{4-r^2} + r \left(1+r^2\right)  \sqrt{4-r^2+4 z^2}}{ \sqrt{1+z^2}}.\end{split}\] Then $\sup_{z>0}g_1(z)=g_1(0)=r \sqrt{4-r^2} + r \left(1+r^2\right)  \sqrt{4-r^2}$.
\end{lemma}

\begin{proof}
 Since $$g_1'(z)=\frac{ r z \left(r^2+r^4-\sqrt{\left(-4+r^2\right) \left(r^2-4 \left(1+z^2\right)\right)}\right)}{\left(1+z^2\right)^{3/2} \sqrt{4-r^2+4 z^2}}$$ and $$(r^2 + r^4)^2 - (-4 + r^2) (r^2 - 4 (1 + z^2))=2 r^6 + r^8 - 16 (1 + z^2) + 4 r^2 (2 + z^2)\le -5 - 12 z^2\le 0,$$ it follows that $g_1'(z)\le 0$ for $z\ge 0$. Thus $g_1(z)\le g_1(0)$ for $z>0$ what we needed to prove.
\end{proof}
\begin{lemma}\label{l4}
Let \[\begin{split}h_2(z)&=\tan^{-1}\left[\frac{r \left(2 r z+
(2-r^2) \sqrt{4-r^2+4 z^2}\right)}{\left(-2+r^2\right)^2-4 \left(-1+r^2\right) z^2}\right]\\&-\tan^{-1}\left[\frac{r \left(2 r z-(2-r^2) \sqrt{4-r^2+4 z^2}\right)}{\left(-2+r^2\right)^2-4 \left(-1+r^2\right) z^2}\right].\end{split}\] Then $\sup_{z>0} h_2(z)=h_2(0)=2\tan^{-1}\left[r\frac{\sqrt{4-r^2}}{2-r^2}\right]$.
\end{lemma}
\begin{proof}
We have $$h_2'(z)=\frac{2r \left(2-r^2\right) z \left(-2+r^2+2 \left(-1+r^2\right) z^2\right)}{\left(1+z^2\right) \sqrt{4-r^2+4 z^2} \left(1+\left(-1+r^2\right)^2 z^2\right)}.$$ So $h_2'(z)\le 0$, and $h_2(z)\le h_2(0)$ for every $z$.
\end{proof}
By using the formula $\frac{1}{2}\log\frac{1+a}{1-a}=\tanh^{-1}(a)$, for $a=\frac{r z \sqrt{4-r^2+4 z^2}}{1+\left(1+r^2\right) z^2}<1$, we obtain
\begin{lemma}\label{l5} For $z>0$ and $0<r<1$ we have
$$h_3(z):= \frac{1}{2}\left(-1+r^2\right) z \tanh^{-1}\left[\frac{r z \sqrt{4-r^2+4 z^2}}{1+\left(1+r^2\right) z^2}\right]\le h_3(0)=0.$$
\end{lemma}
\subsection{Proof of the main result}
\begin{theorem}\label{help}
For $r\in(0,1)$ we have \[\begin{split}\sup_{z>0}{C(z,r)}= C(0,r)= \frac{ \left(r \sqrt{4-r^2} \left(2+r^2\right)+4 \left(1-r^2\right) \tan^{-1}\left[r\frac{\sqrt{4-r^2}}{r^2-2}\right]\right)}{\pi(1+r)r^3}.\end{split}\]
\end{theorem}
\begin{proof}
From Lemmas~\ref{l1},~\ref{l2},~\ref{l3},~\ref{l4} and ~\ref{l5} we obtain  \[\begin{split}
C(z,r)&=\frac{2  (1-r)}{\pi r^3 \sqrt{1+z^2}} (h_1(z)+h_2(z)+h_3(z))\\&\le \frac{2  (1-r)}{\pi r^3 \sqrt{1+z^2}} \left(\frac{r \left(1+r^2\right)  \sqrt{4-r^2+4 z^2}}{2 \left(1-r^2\right) } +\frac{L(z)}{2 \left(1-r^2\right) z}+h_2(0)+h_3(0)\right) \\&\le \frac{2  (1-r)}{\pi r^3 \sqrt{1+z^2}} \left(\frac{r \left(1+r^2\right)  \sqrt{4-r^2+4 z^2}}{2 \left(1-r^2\right) } +\frac{r z \sqrt{4-r^2}}{2 \left(1-r^2\right) z}+h_2(0)+h_3(0)\right)\\&= \frac{  (1-r)}{\pi r^3 (1-r^2)} g_1(z)+\frac{2  (1-r)}{\pi r^3 \sqrt{1+z^2}} \left(h_2(0)+h_3(0)\right)\\ &\le \frac{  (1-r)}{\pi r^3 (1-r^2)} g_1(0)+\frac{2 (1-r)}{\pi r^3} \left(h_2(0)+h_3(0)\right)\\ &= \frac{2  (1-r)}{\pi r^3} \left(h_1(0)+h_2(0)+h_3(0)\right)=C(0,r)\\&= \frac{ \left(r \sqrt{4-r^2} \left(2+r^2\right)+4 \left(1-r^2\right) \tan^{-1}\left[r\frac{\sqrt{4-r^2}}{r^2-2}\right]\right)}{\pi(1+r)r^3}.
\end{split}\]
So $\sup_{z>0}{C(z,r)}= C(0,r)$ what we needed to prove. 
\end{proof}
\begin{remark}
It seems that the same strategy for $n\neq 4$ does not work, because the integrand that appear in definition of the function $C$ is much more complicated. 
\end{remark}


\begin{thebibliography}{10}
\bibitem{AXLER.BOOK}
Sh. Axler, P. Bourdon and W. Ramey, \textit{Harmonic function theory}, Springer, New York (1992).
\bibitem{COLONNA.INDIANA}
F. Colonna,   \textit{The  Bloch  constant  of  bounded  harmonic  mappings},
Indiana University Math. J. \textbf{38} (1989), 829--840.
\bibitem{KALAJ.PAMS}
D. Kalaj and   M. Vuorinen, \textit{On harmonic functions and the Schwarz lemma},
Proc. Amer. Math. Soc. \textbf{140} (2012), 161--165.
\bibitem{KALAJ.POSITIVITY}
D. Kalaj and M. Markovi\'{c}, \textit{Optimal estimates for harmonic functions in the unit ball}, Positivity  \textbf{16} (2012), 771--782.
\bibitem{KALAJ.CAOT}
D. Kalaj and M. Markovi\'{c}, \textit{Optimal Estimates for the Gradient of Harmonic Functions in the Unit Disk},
Complex Analysis and Operator Theory \textbf{7} (2013), 1167--1183.
\bibitem{KHAVINSON.CMB}
D. Khavinson, \textit{An extremal problem for harmonic functions in the ball},
Canadian Math. Bulletin \textbf{35} (1992), 218--220.
\bibitem{KRESIN.JMS}
G. Kresin and V. Maz'ya, \textit{Sharp pointwise estimates for directional derivatives of harmonic functions in a multidimensional ball},
J. Math. Sci. \textbf{169} (2010), 167--187.
\bibitem{KRESIN.DCDS}
G. Kresin  and  V. Maz'ya, \textit{Optimal estimates for the gradient of harmonic functions in the multidimensional half-space},
Discrete Contin. Dyn. Syst. \textbf{28} (2010), 425--440.
\bibitem{book}
G. Kresin  and  V. Maz'ya, \textit{Sharp real-part theorems.} Springer, Berlin, Jan. 1, 2007. (Lecture Notes in Mathematics, 1903). ISBN: 978-3-540-69573-8.
\bibitem{MARKOVIC.INDAG}
M. Markovi\'{c}, \textit{On harmonic functions and the hyperbolic metric}, Indagationes Mathematicae \textbf{26} (2015), 19--23.

\bibitem{MARKOVIC.CA}
M. Markovi\'{c}, \textit{Proof of the Khavinson conjecture near the boundary of the unit ball},  arXiv:1508.00125v1.

\end{thebibliography}
\end{document}